\documentclass[12pt]{amsart}
\usepackage[margin=1.05in]{geometry}                
\usepackage{graphicx}
\usepackage{url}
\usepackage{amssymb}
\usepackage{epstopdf}
\usepackage{color}
\usepackage{subcaption}
\usepackage{comment}
\usepackage{tikz}
\usepackage{algorithm,algpseudocode,algorithmicx}
\usepackage{multirow}
\usepackage{amsmath}
\tikzstyle{vertex}=[circle, draw, inner sep=0pt, minimum size=6pt, fill=black]
\newcommand{\vertex}{\node[vertex]}
\tikzstyle{wvertex}=[circle, draw, inner sep=0pt, minimum size=6pt, fill=white]
\newcommand{\wvertex}{\node[wvertex]}

\DeclareMathOperator{\identity}{Id}
\DeclareMathOperator{\diag}{diag}
\DeclareMathOperator{\rank}{rank}

\newcommand{\calf}{\mathcal{F}}

\newcommand{\CC}{\mathbb{C}}

\newcommand{\rr}{\mathbb{R}}
\newcommand{\cc}{\mathbb{C}}

\newcommand{\nn}{\mathbb{N}}
\newcommand{\kk}{\mathbb{K}}

\numberwithin{equation}{section}

\makeatother

\theoremstyle{plain}
\newtheorem{thm}{Theorem}[section]

\newtheorem{prop}[thm]{Proposition}

\newtheorem*{thm*}{Theorem}
\newtheorem*{lemma*}{Lemma}
\newtheorem*{prop*}{Lemma}
\newtheorem*{cor*}{Corollary}
\newtheorem*{conj*}{Conjecture}

\theoremstyle{definition}
\newtheorem{defn}[thm]{Definition}

\newtheorem*{defn*}{Definition}
\newtheorem{example}[thm]{Example}
\newtheorem{notation}[thm]{Notation}
\newtheorem{remark}[thm]{Remark}
\newtheorem{pr}[thm]{Problem}

\newtheorem{problem}[thm]{Problem}

\theoremstyle{remark}

\newcommand{\HilbertSpace}{H}

\newcommand{\W}{W}

\newcommand{\X}{X}

\newcommand{\Xnr}{\X_{n,r}}

\DeclareMathOperator{\core}{core}
\DeclareMathOperator{\gcore}{\overline{core}}

\usepackage{etoolbox}
\let\bbordermatrix\bordermatrix
\patchcmd{\bbordermatrix}{8.75}{4.75}{}{}
\patchcmd{\bbordermatrix}{\left(}{\left[}{}{}
\patchcmd{\bbordermatrix}{\right)}{\right]}{}{}

\begin{document}
\begin{abstract}
A finite unit norm tight frame is a collection of $r$ vectors in 
$\mathbb{R}^n$ that generalizes the notion of orthonormal bases. 
The affine finite unit norm tight frame variety is the Zariski closure of the set of finite unit norm tight frames.
Determining the fiber of a projection of this variety onto a set of coordinates is called the algebraic finite unit norm tight frame completion problem.
Our techniques involve the algebraic matroid of an algebraic variety, which encodes the dimensions of fibers of coordinate projections.
This work characterizes the bases of the algebraic matroid underlying the variety of  finite unit norm tight frames in $\rr^3$.
Partial results towards similar characterizations for finite unit norm tight frames in 
 $\rr^n$ with $n \ge 4$ are also given. {We provide a method to bound the degree of the projections based off of combinatorial~data.}

\end{abstract}

\title{The algebraic matroid of the finite unit norm tight frame (funtf) variety}




\author{Daniel Irving Bernstein}
\address{Institute for Data, Systems, and Society, Massachusetts Institute of Technology, 77 Massachusetts Avenue, Cambridge, MA 02139}
\email{dibernst@mit.edu}\urladdr{https://dibernstein.github.io}

\author{Cameron Farnsworth}
\address{Department of Mathematics,         Texas State University,  601 University Drive, San Marcos, TX 78666-4684, USA.}
\email[Corresponding author]{clf129@txstate.edu}

\author{Jose Israel Rodriguez}
\address{Department of Mathematics,         University of Wisconsin-Madison,  480 Lincoln Drive, Madison WI 53706-1388, USA.}
\email {jose@math.wisc.edu}\urladdr{http://www.math.wisc.edu/~jose/}

\maketitle

\section{Introduction}
{
A set of unit vectors $\{w_1,\dots,w_r\}$  in 
$\mathbb{R}^n$ is said to be a \emph{finite unit norm tight frame} (or \emph{funtf}) if 
the matrix 
$W:=\left[\begin{smallmatrix} w_{1} & w_{2} & \dots & w_{r}\end{smallmatrix} \right]$
satisfies $\frac{n}{r} WW^T = \identity_n$, where $\identity_n$ is the $n\times n$ identity matrix.
This generalizes the notion of orthonormal bases for $\mathbb{R}^n$.
The Zariski closure (over $\cc$) of the set of $n\times r$ matrices that are finite unit norm tight frames forms an algebraic variety in $\CC^{n\times r}$,
which we  denote by $\Xnr$. The variety $\Xnr$ can be expressed as an algebraic set
(see \cite[p.~4]{CSalgframes} and \cite[Equation 4]{ORSfradeco}) as follows:
  \begin{align}\label{eq:xnrEquations}
	\Xnr =
	\{
	\W\in\CC^{n\times r}& :
	\W\W^T=\frac{r }{n }\identity_n, 
	\,\diag(\W^T\W)=\diag(\identity_r)
	\},
\end{align}
where $\diag(M)$ denotes the diagonal entries of a matrix $M$.
In this paper, we characterize the algebraic matroid underlying $X_{3,r}$  and provide additional partial results for $X_{n,r}$.
}

The remainder of this section describes the applied motivation for our work here and gives a literature review of the area. 
In Section~\ref{sec:algMatroids}, we provide the minimum necessary background on algebraic matroids.
In Section~\ref{sec:reformulate} we set up our notation and collect previous results we will need about $\Xnr$.
In particular, we discuss how matroids arise in connection with algebraic frame completion problems.
In Section~\ref{sec:matroidFuntf}, we provide our main results on the algebraic matroid underlying $\Xnr$  (Theorem~\ref{thm:mustBeConnected}).
Section \ref{sec:degrees} gives a recursive formula for computing the degree of a finite-to-one coordinate projection of $\Xnr$ (Theorem~\ref{thm:degreeOfProjection})
which we then use to completely characterize the degrees of projection onto a basis of~$X_{3,r}$.

\subsection{Algebraic frame theory}
{Frames  generalize the notion of a basis of a vector space and have found use in numerous fields of science and engineering.}
{
Given a Hilbert space $\HilbertSpace$, a \emph{frame} is a set of elements $\{f_k\}_{k\in I}\subset \HilbertSpace$
such that there exist real numbers $A,B$ such that  $0<A\leq B<\infty$ and for every $h\in \HilbertSpace$
\[
A\|h\|^2\leq\sum_{k\in I}\vert\langle h,f_k\rangle\vert^2\leq B\|h\|^2.
\]
}%
These frame conditions are given by Duffin and Schaeffer in \cite[Equation (4)]{DSframeconditions}. If $A=B$, then the frame is called \emph{tight}.
If $\HilbertSpace$ is $n$-dimensional, then any frame has at least $n$ elements.
A frame where each element has norm one is said to be a \emph{unit norm frame}. 
In the literature unit norm frames are also known as normalized frames, uniform frames, and spherical frames. {Throughout this text, we identify the sequence of vectors in a frame with the matrix whose columns are the vectors in the frame. The frames we are interested in are frames of the Hilbert space $\mathbb{R}^n$; however, it is important to keep in mind that we are making a relaxation to a question regarding an algebraic variety in $\mathbb{C}^{n\times r}$.}

A finite frame which is both tight and unit norm is also called {a \emph{finite unit norm tight frame} and is commonly abbreviated in the literature as \emph{funtf}.}
Such frames are the focus of much research because they minimize various measures of error in signal reconstruction \cite{CKframeerasures,Goyal,GVTovercomplete,HolmesPaulsen}.
Algebraic frame theory uses the powerful tools of computational algebraic geometry to solve problems involving finite frame varieties. Such approaches have found success in~\cite{CMSframes,DykemaStrawn,ORSfradeco,Strawnframes}.


Given an $n\times r$ matrix where only a subset of the entries are observed,
the \emph{finite unit norm tight frame completion problem} asks for values of the missing entries such that the resulting completed matrix is a funtf.
The jumping off point for this work is the relaxation of this problem that allows for the missing entries to take on complex values.
We call this relaxation the \emph{algebraic finite unit norm tight frame completion problem}.

Complex frames are also studied where a Hermitian inner product is used, but that is not the focus of this article.
Studying the variety $\Xnr$ in place of the set of finite unit norm tight frames gives one access to tools from algebraic geometry,
and results about $\Xnr$ can lead to insight about the set of finite unit norm tight frames---see for example \cite{CSalgframes}.
Many works have studied the properties of various sets of frames considered as varieties.
{For example, dimensions of $(\mu,S)$-frame varieties, which are spaces of matrices $W=[w_1\cdots w_r]$, real or complex, satisfying $WW^*=S$
for some Hermitian (symmetric) positive definite matrix $S$ such that $\|w_k\|=\mu_k$, were considered in \cite{StrawnMasterThesis}. Finite unit norm tight frames are a special case of these $(\mu,S)$-frames where $\mu_i=1$ and $S$ is a scalar multiple of the identity matrix.} Along with the fundamental groups, the dimensions of finite unit norm tight frame varieties were derived in \cite{DykemaStrawn}. 

 In \cite{Strawnframes}, nonsingular points of $(\mu,S)$-frames are characterized along with the tangent spaces at these nonsingular points on these varieties. The connectivity of the finite unit norm tight frame variety along with its irreducibility are studied in \cite{CMSframes}. In \cite{HagaPegel}, the polytope of eigensteps of finite equal norm tight frames is studied.
{ These \emph{eigensteps} are sequences of interlacing spectra used by \cite{CFMPSframes} to construct finite frames of prescribed norms}
and the dimension of finite unit norm tight frame varieties is noted to be related to the dimension of these polytopes.

\subsection{Algebraic finite unit norm tight frame completion}\label{ss:afc}
Due to their robustness to erasures and additive noise, unit norm tight frames play an important role in signal processing.
Explicit constructions for unit norm tight frames are quite recent despite theoretical work regarding existence being quite classical.
{The \emph{Schur-Horn Theorem}~\cite{Horn,Schur} characterizes the pairs $(\lambda,\mu)$ such that there exists
a frame whose frame operator has spectrum $\lambda$ and lengths $\mu$.}
However, explicit constructions for these frames have remained scarce.
In \cite{GKKframeerasures} the authors give a constructive characterization for all unit norm tight frames in $\rr^2$
and provide a construction technique known as \emph{harmonic frames} for {unit norm tight} frames in $\rr^n$.
An alternative constructive technique called \emph{spectral tetris} is given in \cite{CFMWZfusionframes}.
An explicit construction of every unit norm tight frame was finally given by \cite{CFMPSframes,FMPSframes}.

The previous paragraph covers results on explicitly constructing frames with prescribed spectrum and whose vectors' lengths are prescribed.
However, what if you have specific vectors you want included in your frame? How do you complete this partial set of vectors into a tight frame?
The work
 \cite{FWWgeneratingframes} answers how many vectors must be added to complete your set of vectors into a tight frame, and in the case when all vectors are unit norm, they also provide a lower bound (which is not sharp) for the number of vectors required to complete the set of vectors into a tight frame.  The minimum number of vectors needed to add to your set of vectors to complete it to a frame when their norms are prescribed is provided by \cite{MRframecompletions}. In both papers, it is assumed that you start with a set of vectors. {See also~\cite{MR3473144}, in which the authors characterize the spectra of all frame operators of frames completed from the addition of extra vectors to another frame and discuss completion in such a way as to produce a frame minimizing the condition number, the mean squared reconstruction error, and the frame potential.} 

In this paper, we take a different approach than Feng, Wang, and Wang, {Massey and Ruiz, or Fickus, Marks, and Poteet.} Instead of starting with a set of vectors and asking how many more vectors are needed to have a tight frame,  we have the following generalization of the problem.

\begin{problem}\label{pr:funtfCompletion}
Given a partially observed $n\times r$ matrix,
determine if the missing entries can be completed such that the columns form a finite unit norm tight frame.
\end{problem}

{
A generalization of Problem~$\ref{pr:funtfCompletion}$ is to ask how many completions there are. 

\begin{problem}\label{pr:funtfDegree}
Given some known entries of an  $n\times r$ matrix, 
find the cardinality of the number of completions 
such that the columns form a finite unit norm tight frame.
\end{problem}

When studying the algebraic geometry of low-rank matrix completion,
one often studies the analogue of Problems \ref{pr:funtfCompletion} and \ref{pr:funtfDegree},
where instead of completing to a finite unit norm tight frame,
one completes to a matrix of a particular specified low rank \cite{kiraly2015algebraic,blekherman2019maximum,bernstein2017completion,bernstein2019typical,bernstein2020typical}.

Problem \ref{pr:funtfDegree} can be studied using algebraic geometry 
by considering the known entries as a set of defining equations of an affine linear space $\mathcal{L}$ in $\mathbb{C}^{n\times r}$ and then studying the degree of 
$\Xnr\cap \mathcal{L}$.
When  $\Xnr\cap \mathcal{L}$ is finite a set of points,
the degree of $\Xnr\cap \mathcal{L}$ is the number of points, counted with multiplicity,
and is an upper bound for the number of completions.  
The degree of $\Xnr\cap \mathcal{L}$ need not coincide with the degree of $\Xnr$,
but when the linear space $\mathcal{L}$ is generic such that $\Xnr\cap \mathcal{L}\neq \emptyset$, 
the degree of $\Xnr\cap \mathcal{L}$ and $\Xnr$ are the same. 
For $n=r$, the variety $\Xnr$ is the orthogonal group and the degree of this variety was determined 
in \cite{BBBKR2017} using representation theory.
An analogous question from rigidity theory
asks for the number of realizations of a rigid graph.
This same algebraic approach for getting an upper bound works in this context,
this time by computing or bounding the degree of (coordinate projections of) the Cayley-Menger variety \cite{jackson2012number,emiris2009algebraic,capco2018number,borcea2004number}.

After solving Problem \ref{pr:funtfCompletion},
the most natural next problem is to actually compute such a completion, if it exists.
One can attempt to do this using Gr\"obner bases algorithms, but this is likely to be prohibitively slow.
A faster way to do this would be using numerical algebraic geometry
using e.g. \texttt{Bertini} \cite{BHSW06}, \texttt{PHCPack} \cite{verschelde1999algorithm}, or \texttt{Macaulay2} \cite{MR2881262,M2}.
Knowing the degree of coordinate projections of $\Xnr$ could be helpful in this approach.
One could could also try using the combinatorics of the set of known entries
to derive polynomials in the missing entries that could be solved to complete
the finite unit norm tight frame.
In the language of algebraic matroids,
this is the problem of finding combinatorial descriptions of the circuit polynomials of $\Xnr$.
This would be an interesting future direction.
}

\section{The basics of algebraic matroids}\label{sec:algMatroids}
We now take a detour to introduce the minimum necessary background on algebraic matroids.
Since the only matroids considered in this paper will be algebraic,
we will not discuss or define abstract matroids.
Moreover, our study will be limited to those that are algebraic over $\rr$ or $\cc$.
The reader who is interested in learning about more general (algebraic) matroids is advised to consult the textbook \cite{oxley2006matroid}.

Let $\kk$ be a field.
Given a finite set $E$,
we let $\kk^E$ denote the vector space whose coordinates are indexed by the elements of $E$.
Each subset $S \subseteq E$ of coordinates is associated with the linear projection $\pi_S: \kk^E \rightarrow \kk^S$
that sends each point $(x_e)_{e \in E}$ to $(x_e)_{e \in S}$.
The ring of polynomials with coefficients in $\kk$ and indeterminants indexed by $E$ will be denoted $\kk[x_e: e \in E]$
and the corresponding field of rational functions will be denoted $\kk(x_e: e \in E)$.
The ideal in $\kk[x_e: e \in E]$ generated by a finite set of polynomials $f_1,\dots,f_k \in \kk[x_e : e \in E]$
will be notated as $(f_1,\dots,f_k)$.
Given a set $X \subseteq \kk^E$, we let $I(X)$ denote the ideal of all polynomial functions that vanish on $X$.
\begin{defn}\label{defn:algebraicMatroid}
	Let $E$ be a finite set, let $\kk$ be $\rr$ or $\cc$, and let $X \subseteq \kk^E$ be an irreducible variety.
    A subset of coordinates $S \subseteq E$ is
    \begin{enumerate}
        \item \emph{independent} in $X$ if $I(\pi_S(X))$ is the zero ideal
        \item \emph{spanning} in $X$ if $\dim(\pi_S(X)) = \dim(X)$
        \item\label{item:basis} a \emph{basis} of $X$ if $S$ is both independent and spanning.
    \end{enumerate}
    Any one of the three set systems consisting of the independent sets,
    the spanning sets, or the bases of an irreducible variety determines the other two.
    The combinatorial structure specified by any one of these set systems is called the \emph{algebraic matroid underlying $X$}.
\end{defn}

\begin{example}
    Let $X \subset \rr^{[4]}$ be the linear variety defined by the vanishing of the linear forms $x_1 - 5x_2 = 0$ and $x_3+2x_4 = 0$.
    The bases of $X$ are
    \[
        \{1,3\}, \quad \{1,4\}, \quad \{2,3\}, \quad \{2,4\}.
    \]
    The independent sets of $X$ are the subsets of the bases,
    and the spanning sets are the supersets.
    Moreover, all the bases have cardinality two, which is also the dimension of $X$.
    This is not a coincidence---see Proposition \ref{prop:basesSameSize} below.
\end{example}

We now describe the intuition behind the algebraic matroid underlying an irreducible variety $X \subseteq \kk^E$.
When $S \subseteq E$ is independent,
the coordinates $(x_e)_{e \in S}$ can be given arbitrary generic values, and the resulting vector can be completed
to a point in $X$.
When $S \subseteq E$ is spanning and $x \in X$ is generic,
then the coordinates $(x_e)_{e\notin S}$ can be determined by solving a zero-dimensional system of polynomials
whose coefficients are polynomials in $(x_e)_{e \in S}$.
In other words, the set $\pi_S^{-1}(\pi_S(x)) \cap X$ is generically finite.

It's important to recall Definition~\ref{defn:algebraicMatroid} requires  $X$ to be an irreducible variety.
This ensures that the algebraic matroid underlying $X$ is indeed a matroid (see for example \cite[Proposition 1.2.9]{bernstein2018matroids}).
Proposition \ref{prop:basesSameSize} below then follows from the fact that all bases of a matroid have the same size \cite[Chapter 1]{oxley2006matroid}.

\begin{prop}\label{prop:basesSameSize}
{For any basis $B\subseteq E$ of an irreducible variety $X\subseteq \mathbb{K}^E$, the cardinality of the set $B$ equals $\dim(X)$.}
\end{prop}

Given finite sets $A$ and $B$ and a field $\kk$, we let $\kk^{A\times B}$ denote the set of matrices with entries in $\kk$
whose rows are indexed by elements of $A$ and whose columns are indexed by elements of $B$.
Given polynomials $f_1,\dots,f_k\in \kk[x_e: e \in E]$
the \emph{Jacobian matrix of $f_1,\dots,f_k$} is the matrix $J(f_1,\dots,f_k) \in (\kk(x_e: e \in E))^{[k]\times E}$
whose $(i,e)$ entry is the partial derivative $\frac{\partial f_i}{\partial x_e}$.
In Section~\ref{sec:matroidFuntf}, we will often work with submatrices of a Jacobian matrix.
For this reason, we introduce the following notation.
\begin{notation}
Let $M$ denote a matrix whose columns are indexed by a set $E$.
The submatrix of a given $M$ with columns corresponding to the elements of a subset $S$ of $E$
is denoted~$M_{S}$.
\end{notation}

The following proposition is useful for computing the bases of the algebraic matroid underlying a given irreducible variety.
It is well known, and usually stated in terms of matroid duals.
We state it here in more elementary terms for the purposes of keeping the necessary matroid theory background at a minimum.
{For more details, see \cite[Section 2.2]{rosen2014computing} and \cite[Proposition 6.7.10]{oxley2006matroid}.}

\begin{prop}\label{prop:useJacobianToDetermineIfBasis}
    Let $E$ be a finite set, let $\kk$ be $\rr$ or $\cc$ and let $X \subseteq \kk^E$ be an irreducible variety of dimension $d$
    such that $I(X) = (f_1,\dots,f_k)$.
    A subset $S \subseteq E$ of size $d$ is a basis of $X$ if and only if the rank of $J(f_1,\dots,f_k)_{E\setminus S}$ is $|E|-d$.
\end{prop}

\section{Algebraic matroids to algebraic funtf completion}\label{sec:reformulate}
\renewcommand{\rr}{r}
\renewcommand{\nn}{n}

Recall from Section~\ref{eq:xnrEquations} that the Zariski closure in $\CC^{\nn\times\rr}$ of the
set of
$\nn\times\rr$ matrices  such that the columns form a finite unit norm tight frame is
denoted by $\Xnr$. In other words,
this paper studies the following algebraic relaxation of Problem \ref{pr:funtfCompletion}.
\begin{pr}[The algebraic frame completion problem]\label{pr:algebraicFuntfCompletion}
	Given some known entries of an $n\times r$ matrix,
	determine if the matrix can be completed to an element of $\Xnr$.
\end{pr}

We say $\Xnr$ is an \emph{affine finite unit norm tight frame variety} and call a matrix in $\Xnr$ a \emph{finite unit norm tight frame (funtf) matrix}.
The  $\binom{n+1}{2} + r$ { scalar equations defining $\Xnr$ shown in \eqref{eq:xnrEquations}
were found in \cite[p.~4]{CSalgframes} and \cite[Equation 4]{ORSfradeco}.} \color{black}
We will express the polynomials defining the affine finite unit norm tight frame variety in the ring $\kk[x_{ij}: 1\le i \le n, 1 \le j \le r]$ where $x_{ij}$ will represent the $ij$ entry of a matrix.
\newcommand{\normPoly}{g}
Indeed, the column norm constraints on $W$ can be expressed as the following $\rr$ polynomials set to zero:
\begin{equation}\label{eqG}
(\normPoly_1,\dots,\normPoly_\rr) :=\diag(\W^T\W-\identity_\rr),
\end{equation}
{while the orthogonal row constraints on $W$ can be expressed as the following $\binom{n+1}{2}$ polynomials $f_{ij}$, $i\leq j$:
\begin{equation}\label{eqF}
\begin{bmatrix}
f_{11} &f_{12}&\ldots&f_{1n}\\
f_{12} &f_{22}&\ldots&f_{2n}\\
\vdots& 		&\ddots	&\vdots	\\
f_{1n} &f_{2n}&\ldots&f_{nn}
\end{bmatrix} : =WW^T-\frac{\rr}{\nn}\identity_\nn.
\end{equation}
}
{
    Elementary algebra shows that these polynomials satisfy the following relation
    	\begin{equation}\label{eq:fgRelation}
	 \sum_{i=1}^n f_{i,i} = \sum_{j=1}^r g_j = ||W||_F^2-r,
	\end{equation}
	where $||\cdot ||_F$ denotes the Frobenius norm.
}

The problem of algebraic finite unit norm tight frame completion can be cast as the problem of projecting an affine finite unit norm tight frame variety to a subset of coordinates.
Let $E \subseteq [n]\times [r]$ denote a subset of coordinates of $\CC^{\nn\times\rr}$.
We will think of $E$ as indexing ``known'' entries,
and the algebraic finite unit norm tight frame completion problem is to determine the remaining ``unknown'' entries
so that the completed matrix is a finite unit norm tight frame.
Let $\pi_E$
denote the respective coordinate projection.
The algebraic finite unit norm tight frame completions of a given $M \in \cc^E$ are the elements of the fiber $\pi^{-1}_E(M)$.
It follows that $E$ is independent in $\Xnr$ if and only if every generic $M \in \cc^E$ has an algebraic finite unit norm tight frame completion
and that $E$ is spanning in $\Xnr$ if and only if each nonempty fiber $\pi^{-1}(\pi_E(M))$ is finite when $M$ is generic.
Thus in the generic case, Problem \ref{pr:algebraicFuntfCompletion} is equivalent to the following.

\begin{pr}\label{pr:motivatingProblem}
    Find a combinatorial description of the algebraic matroid underlying $\Xnr$
\end{pr}

The first steps towards solving Problem \ref{pr:motivatingProblem} are determining the irreducibility and dimension of $\Xnr$.
Fortunately, this was done in \cite[Theorem 4.3(ii)]{DykemaStrawn}, \cite[Corollary 3.5]{Strawnframes}, {and \cite[Theorem 1.4]{CMSframes} which we summarize in the following theorem.}

\begin{thm}\label{thm:dimOfFuntf}
    The dimension of the affine finite unit norm tight frame variety $\Xnr$ is
    \begin{equation}\label{eq:implicitDefinition}
        \dim(\Xnr) = nr-\binom{n+1}{2} - r + 1\qquad{\rm provided} \ r>n\ge 2.
    \end{equation}
    It is irreducible when $r \ge n+2 > 4$.
\end{thm}

In our work, we look to determine each basis (Definition~\ref{defn:algebraicMatroid}, item \ref{item:basis}) of $\Xnr$.
We restrict our study to $r \ge n+2 > 4$ so that $\Xnr$ is irreducible and thus gives a matroid.
We seek a combinatorial description using bipartite graphs.
Bipartite graphs provide a natural language for attacking Problem~\ref{pr:motivatingProblem}.
Given finite sets $A$ and $B$ and a subset $S \subseteq A \times B$, we let $(A,B,S)$ denote the \emph{bipartite
graph} with partite vertex sets $A$ and $B$ and edge set $S$.
We call two bipartite graphs $(A_1,B_1,S_1)$ and $(A_2,B_2,S_2)$ \emph{bipartite isomorphic} if there exists
a graph isomorphism $\phi: A_1 \cup B_1 \rightarrow A_2 \cup B_2$ such that $\phi(A_1) = A_2$ and $\phi(B_1) = B_2$.

Every subset $E$ of entries of an $n\times r$ matrix can be identified with the bipartite graph $([n],[r],[n]\times [r] \setminus E)$,
which we denote by $G_E$.
The edges of $G_E$ are in bijection with the complement of $E$ and not $E$ itself.
This stands in contrast to what is often done in the algebraic matrix completion literature,
but will make our results much cleaner to state.
Neither row-swapping nor column-swapping affects whether a given subset $E$ of entries of an $n\times r$ matrix
is an independent set (or a basis, or spanning set) of $\Xnr$.
Therefore, whether a given subset $E$ of entries is independent (or a basis or spanning) in $\Xnr$ only depends
on the bipartite isomorphism equivalence class of $G_E$.
The (non-bipartite) graph isomorphism class of $G_E$ may not be sufficient to determine
whether $E$ is independent (or a basis or spanning) in $\Xnr$ because the transpose of a finite unit norm tight frame matrix $\W$ may not be funtf.
So from now on, we will only consider bipartite graphs up to their bipartite isomorphism classes.
We may now phrase Problem \ref{pr:motivatingProblem} more concretely as follows.

\begin{pr}\label{pr:mainProblemWithGraphs}
    For which (bipartite isomorphism classes of) bipartite graphs $G_E$ is $E$ a basis of $\Xnr$?
\end{pr}

We will sometimes find it useful to represent a subset $E \subseteq [n]\times [r]$
as the $\{0,1\}$-matrix whose $ij$ entry is $1$ if $(i,j) \in E$ and $0$ otherwise.
Such a representation will be called a \emph{matrix entry representation}.

\begin{figure}[htbp!]
    \begin{equation*}
        \begin{tikzpicture}
            \vertex (r1) at (0,0)[label=left:$1$]{};
            \vertex (r2) at (0,-1)[label=left:$2$]{};
            \vertex (r3) at (0,-2)[label=left:$3$]{};
            \vertex (c1) at (1,1)[label=right:$1$]{};
            \vertex (c2) at (1,0)[label=right:$2$]{};
            \vertex (c3) at (1,-1)[label=right:$3$]{};
            \vertex (c4) at (1,-2)[label=right:$4$]{};
            \vertex (c5) at (1,-3)[label=right:$5$]{};
            \path
                (r1) edge (c1)
                (r1) edge (c4)
                (r1) edge (c5)
                (r2) edge (c2)
                (r2) edge (c4)
                (r2) edge (c5)
                (r3) edge (c3)
                (r3) edge (c4)
                (r3) edge (c5)
            ;
            \node at (5,-1){
            $\begin{pmatrix}
            0&1&1&0&0\\
            1&0&1&0&0\\
            1&1&0&0&0
        \end{pmatrix}$};
        \end{tikzpicture}
    \end{equation*}
    \caption{Depicting $E := \{(1,2),(1,3),(2,1),(2,3),(3,1),(3,2)\} \subset [3]\times[5]$
    as the bipartite graph $G_E$ and as a $\{0,1\}$-matrix.}
\end{figure}
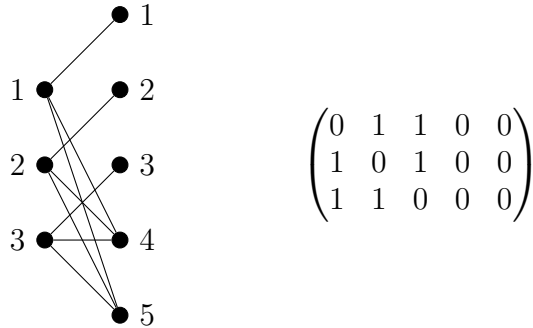


\section{The algebraic matroid underlying the finite unit norm tight frame variety}\label{sec:matroidFuntf}
In this section we give combinatorial criteria on the bases of $\Xnr$.
First we show that if $E$ is a basis of $\Xnr$, then the graph $G_E$ is connected.
Moreover, when $n=3$ the converse is true as well.
Second,
we show that whether or not $E$ is a basis of $\Xnr$ only depends on the  $2$-core of $G_E$.
This allows us to determine a combinatorial criterion for every $\rr$ after fixing~$\nn$.

\subsection{Graph connectivity}
We begin with some graph theoretic definitions.
Let $G = (A,B,S)$ be a bipartite graph.
The \emph{greater $2$-core of $G$}, denoted $\gcore_2(G)$, is the graph obtained from $G$ by iteratively removing all edges that are incident to a vertex of degree one.
The \emph{$2$-core of $G$}, denoted $\core_2(G)$, is the graph obtained by deleting the isolated vertices from $\gcore_2(G)$.
Figure \ref{fig:2coreExample} shows a graph alongside its greater $2$-core and its $2$-core.

\begin{figure}[htb!]
    \begin{tikzpicture}
        \node at (-1,1){$G = $};
        \vertex (1) at (0,0){};
        \vertex (2) at (0,1){};
        \vertex (3) at (0,2){};
        \vertex (a) at (1,0){};
        \vertex (b) at (1,1){};
        \vertex (c) at (1,2){};
        \path
            (3) edge (c) edge (b)
            (2) edge (a) edge (b)
            (1) edge (a) edge (b)
        ;
    \end{tikzpicture}
    \qquad \
    \begin{tikzpicture}
        \node at (-1.5,1){$\gcore_2(G) = $};
        \vertex (1) at (0,0){};
        \vertex (2) at (0,1){};
        \vertex (3) at (0,2){};
        \vertex (a) at (1,0){};
        \vertex (b) at (1,1){};
        \vertex (c) at (1,2){};
        \path
            (1) edge (a) edge (b)
            (2) edge (a) edge (b)
        ;
    \end{tikzpicture}
    \qquad \
    \begin{tikzpicture}
        \node at (-1.5,1){$\core_2(G) = $};
        \vertex (1) at (0,0){};
        \vertex (2) at (0,1){};
        \vertex (a) at (1,0){};
        \vertex (b) at (1,1){};
        \path
            (1) edge (a) edge (b)
            (2) edge (a) edge (b)
        ;
    \end{tikzpicture}
    \caption{A bipartite graph $G$ alongside its greater $2$-core and its $2$-core.}
    \label{fig:2coreExample}
\end{figure}
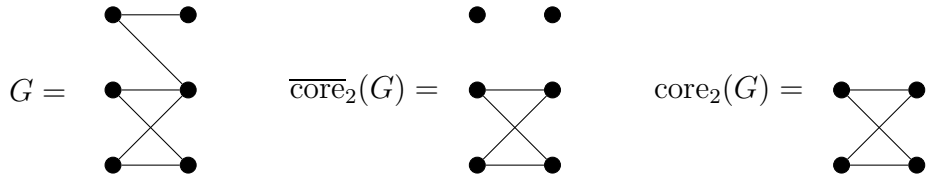

{
Since the graph $G$ need not be connected, there may not exist a spanning tree for $G$. Instead,  consider a \emph{spanning forest} for $G$, which is a maximal acyclic subgraph of $G$, or equivalently a subgraph consisting of a spanning tree in each connected component.   }
{
A \emph{circuit} is a nonempty trail in $G$ such that the first and last vertex coincide, or equivalently, a non-empty sequence $(e_1,\dots,e_k)$ of edges in $G$
for which there is a sequence of vertices $(v_1,\dots,v_k,v_1)$
such that for  $i=1,\dots,k-1$ the vertices of $e_i$ are $(v_i,v_{i+1})$
and the vertices of $e_k$ are $(v_k,v_1)$. 
A cycle is an example of a circuit. 
}
Given a spanning forest $F$ of $G$
and an edge $e$ of $G$ not appearing in $F$,
the graph $F \cup \{e\}$ has exactly one cycle which must contain $e$.
This cycle is called \emph{the fundamental circuit of $e$ with respect to $F$}.

Given a subset $S'\subseteq S$ of the edge set of $G$, the \emph{characteristic vector of $S'$}
is the vector in $\{0,1\}^S$  that has ones at entries corresponding to elements of $S'$
and zeros at all other entries.
The \emph{incidence matrix of $G$} is the matrix whose rows are indexed by the vertices of $G$,
and the row corresponding to a vertex $v$ is the characteristic vector of the set of edges that are incident to $v$.
The columns of the incidence matrix of $G$ are indexed by the edges of $G$.

\begin{example}\label{ex:K}
We use the notation $K_{a,b}$ to denote the complete bipartite graph on partite sets of size $a$ and $b$.
The incidence matrix of $K_{3,5}$ is given by the $8\times 15$ matrix below
\[
\left(\begin{array}{ccccccccccccccc}
1 & 1 & 1 & 1 & 1 & 0 & 0 & 0 & 0 & 0 & 0 & 0 & 0 & 0 & 0\\
0 & 0 & 0 & 0 & 0 & 1 & 1 & 1 & 1 & 1 & 0 & 0 & 0 & 0 & 0\\
0 & 0 & 0 & 0 & 0 & 0 & 0 & 0 & 0 & 0 & 1 & 1 & 1 & 1 & 1\\
\hline
1 & 0 & 0 & 0 & 0 & 1 & 0 & 0 & 0 & 0 & 1 & 0 & 0 & 0 & 0\\
0 & 1 & 0 & 0 & 0 & 0 & 1 & 0 & 0 & 0 & 0 & 1 & 0 & 0 & 0\\
0 & 0 & 1 & 0 & 0 & 0 & 0 & 1 & 0 & 0 & 0 & 0 & 1 & 0 & 0\\
0 & 0 & 0 & 1 & 0 & 0 & 0 & 0 & 1 & 0 & 0 & 0 & 0 & 1 & 0\\
0 & 0 & 0 & 0 & 1 & 0 & 0 & 0 & 0 & 1 & 0 & 0 & 0 & 0 & 1
\end{array}\right).
\]
This matrix is naturally partitioned via the vertices in each partite of the graph $K_{3,5}$.

\end{example}
\begin{thm}\label{thm:mustBeConnected}
    Assume $r \ge n+2 > 4$
    and let $E \subseteq [n]\times [r]$ have cardinality $nr - \binom{n+1}{2} - r + 1$.
    If $E$ is a basis of $\Xnr$,
    then $G_E$ is connected.
    When $n=3$, the converse is true as well.
\end{thm}
\begin{proof}
Let $g_i$ and $f_{ij}$ denote the polynomials as in \eqref{eqG} and $\eqref{eqF}$; this set of polynomials generate the ideal of $\Xnr$.
	Let $J$ be the $(r+\binom{n+1}{2}) \times nr$ Jacobian matrix
    \begin{equation*}
        J:=J(g_1,\dots,g_r,f_{11},f_{12}\dots,f_{nn}).
    \end{equation*}
	Proposition~\ref{prop:useJacobianToDetermineIfBasis} implies that
    $E$ is a basis of $\Xnr$ if and only if
    the $(r+\binom{n+1}{2})\times (r+\binom{n+1}{2}-1)$
    matrix $J_{[n]\times[r]\setminus E}$ has full rank,  
     or equivalently, 
     \[
     \rank(J_{[n]\times[r]\setminus E})=r +\binom{n+1}{2} - 1.
     \]

    Let $J'$ and $J''$ denote the following matrices of size $(r+\binom{n+1}{2})\times nr$ and $(r+n)\times nr$, respectively
	\begin{align*}
	    J'&:=J \cdot \diag({1}/{x_{11}},1/x_{12}\dots,{1}/{x_{nr}}),\\
	    J''&:=J(g_1,\dots,g_r,f_{11},f_{22},\dots,f_{ii},\dots,f_{nn})\cdot \diag({1}/{x_{11}},1/x_{12}\dots,{1}/{x_{nr}}),
	\end{align*}
    where $\diag(w)$ denotes the matrix with the vector $w$ along its diagonal.
{
The matrix $J''$ is twice the incidence matrix of $K_{n,r}$ from Example~\ref{ex:K},
} and $J''_{[n]\times [r]\setminus E}$
    is twice the incidence matrix of $G_E$.
    Thus if $G_E$ has $c$ connected components, then
    \begin{equation*}
    	\rank (J''_{[n]\times [r]\setminus E})= r+n-c.
    \end{equation*}
    On the other hand, since $J'$ can be obtained from $J''$ by including $\binom{n}{2}$ additional rows, we have
    \begin{equation*}
	 \rank(J'_{[n]\times [r]\setminus E}) \le
	 \rank (J''_{[n]\times [r]\setminus E})+\binom{n}{2}
	 =    r + \binom{n+1}{2}-c.
    \end{equation*}
Since $\rank(J_{[n]\times [r]\setminus E}) = \rank(J'_{[n]\times [r]\setminus E})$,
we have that if $G_E$ is disconnected, then $J_{[n]\times [r]\setminus E}$ is rank deficient and thus $E$ is not a basis.

	Now, having proved that $E$ being a basis implies connectivity of $G_E$,
    we assume that $n = 3$ and prove the converse.
    Further assume that $G_E$ is connected with $\binom{n+1}{2}+r-1=r+2$ edges.
    We will show that $E$ is a basis of $\Xnr$ by showing that $J'_{[n]\times[r]\setminus E}$ has full rank.
    	This is done by splitting $J'_{[n]\times[r]\setminus E}$ into two row submatrices whose kernels intersect trivially.

    Twice the incidence matrix of the complete bipartite graph $K_{3,r}$ is a row-submatrix of $J'$.
    Therefore, any linear relation among the columns of $J'$ must lie in the linear space
    \begin{equation*}
    	\cc\{v_C: C \textnormal{ is a circuit of } K_{3,r}\} 
    \end{equation*}
    where $v_C \in \cc^{[3]\times [r]}$ is the $\{1,-1,0\}$-vector obtained from the characteristic vector of $C$
    by
    giving adjacent edges opposite signs.
    The row of $J'$ corresponding to the constraint $f_{ab}=0$ with $a \neq b$
    has $x_{bi}/x_{ai}$ at the column corresponding to $x_{ai}$ and $x_{ai}/{x_{bi}}$ at the column corresponding to $x_{bi}$.
    For ease of notation, we introduce the change of variables
    \begin{equation*}
        t_{1i} := x_{2i}/x_{1i} \qquad t_{2i} := x_{3i}/x_{1i}.
    \end{equation*}
    With this change of variables, the rows of $J'$ corresponding to the constraints $f_{12}=f_{13}=f_{23}=0$
    form the matrix $K$ shown below
    \begin{equation*}
        K:=\bbordermatrix{
            &x_{11} & \dots & x_{1r} & x_{21} & \dots & x_{2r} & x_{31} & \dots & x_{3r} \cr
            f_{12} &t_{11} & \dots & t_{1r} & t^{-1}_{11} & \dots & t^{-1}_{1r} & 0 & \dots & 0 \cr
            f_{13} &t_{21} & \dots & t_{2r} & 0 & \dots & 0 & t^{-1}_{21} & \dots & t^{-1}_{2r} \cr
            f_{23} &0 & \dots & 0 & t_{21}t^{-1}_{11} & \dots & t_{2r}t^{-1}_{1r} & t_{11}t^{-1}_{21} & \dots & t_{1r}t^{-1}_{2r}
        }.
    \end{equation*}
    Fix a spanning tree $T$ of $G_E$ and let $e_1,e_2,e_3$ denote the three edges of $G_E$ that are not contained in $T$.
    Let $C_i$ denote the fundamental circuit of $e_i$ with respect to $T$.
    The space of linear relations among the columns of $J'$ corresponding to the edges of $G_E$
    lies within the three-dimensional subspace $\cc\{v_{C_i}: i = 1,2,3\}$.
    We now show that no nonzero element of $\cc\{v_{C_i}: i = 1,2,3\}$ lies in the kernel of $K$.
    It will then follow that the column-submatrix of $J'$ corresponding to the edges of $G_E$ has maximum rank.

    The three fundamental circuits $C_1,C_2,$ and $C_3$ all lie in $\core_2(G_E)$
    which is a bipartite graph on partite sets of size $n'\le 3$ and $r' \le r$.
    Each vertex of $\core_2(G_E)$ has degree at least $2$, so $n',r'\ge 2$.
    Since $G_E$ is connected, $\core_2(G_E)$ must also be connected.
    Hence since $C_1,C_2,$ and $C_3$ all lie in $\core_2(G_E)$,
    $\core_2(G_E)$ must have exactly $n'+r'+2$ edges.
    Since each vertex has degree at least $2$, $2r' \le n'+r'+2$ and so $r' \le n'+2$.
    So thus far, we only need to consider
    $(n',r') = (2,2),(2,3),(2,4),(3,2),(3,3),(3,4),(3,5)$.
    Among these, the only $(n',r')$-pairs such that there even exists such a bipartite graph with the correct number of edges
    are $(2,4),(3,3),(3,4),(3,5)$.
    For these values of $(n',r')$,
    we may compute all the connected bipartite graphs on partite sets of size $n'$ and $r'$
    with minimum degree $2$ and exactly $n'+r'+2$ edges using the \texttt{genbg} command of \texttt{Nauty and Traces} \cite{McKay201494}.
    There are seven such graphs and they are displayed in Figure \ref{fig:sevenPossible}
    with vertices labeled according to which row or column they correspond to.
    \color{black}
    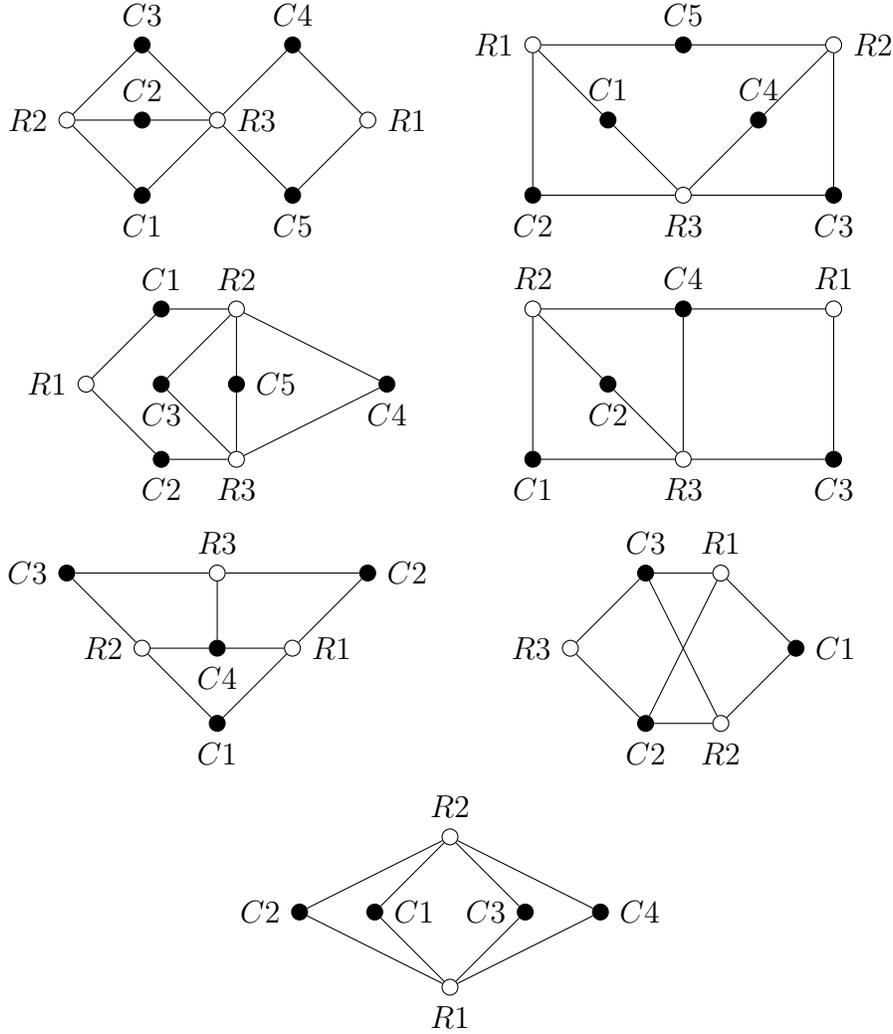
\begin{figure}
    \centering
    \begin{tabular}{cc}
        \begin{tikzpicture}
            \wvertex (R2) at (0,0)[label=left:$R2$]{};
            \vertex (C2) at (1,0)[label=above:$C2$]{};
            \vertex (C3) at (1,1)[label=above:$C3$]{};
            \vertex (C1) at (1,-1)[label=below:$C1$]{};
            \wvertex (R3) at (2,0)[label=right:$R3$]{};
            \wvertex (R1) at (4,0)[label=right:$R1$]{};
            \vertex (C4) at (3,1)[label=above:$C4$]{};
            \vertex (C5) at (3,-1)[label=below:$C5$]{};
            \path
            	(R2) edge (C1) edge (C2) edge (C3)
            	(R3) edge (C1) edge (C2) edge (C3) edge (C4) edge (C5)
            	(R1) edge (C4) edge (C5)
            ;
        \end{tikzpicture}&
        \begin{tikzpicture}
            \vertex (C2) at (0,0)[label=below:$C2$]{};
            \wvertex (R3) at (2,0)[label=below:$R3$]{};
            \vertex (C3) at (4,0)[label=below:$C3$]{};
            \wvertex (R1) at (0,2)[label=left:$R1$]{};
            \wvertex (R2) at (4,2)[label=right:$R2$]{};
            \vertex (C1) at (1,1)[label=above:$C1$]{};
            \vertex (C4) at (3,1)[label=above:$C4$]{};
            \vertex (C5) at (2,2)[label=above:$C5$]{};
            \path
            	(R3) edge (C1) edge (C2) edge (C3) edge (C4)
            	(R1) edge (C1) edge (C2) edge (C5)
            	(R2) edge (C3) edge (C4) edge (C5)
            ;
        \end{tikzpicture}\\
        \begin{tikzpicture}
            \wvertex (R1) at (0,0)[label=left:$R1$]{};
            \vertex (C1) at (1,1)[label=above:$C1$]{};
            \vertex (C2) at (1,-1)[label=below:$C2$]{};
            \vertex (C5) at (2,0)[label=right:$C5$]{};
            \vertex (C3) at (1,0)[label=below:$C3$]{};
            \vertex (C4) at (4,0)[label=below:$C4$]{};
            \wvertex (R2) at (2,1)[label=above:$R2$]{};
            \wvertex (R3) at (2,-1)[label=below:$R3$]{};
            \path
            	(R1) edge (C1) edge (C2)
            	(C3) edge (R2) edge (R3)
            	(C4) edge (R2) edge (R3)
            	(C5) edge (R2) edge (R3)
            	(C1) edge (R2)
            	(C2) edge (R3)
            ;
        \end{tikzpicture}&
        \begin{tikzpicture}
            \vertex (C1) at (0,0)[label=below:$C1$]{};
            \wvertex (R3) at (2,0)[label=below:$R3$]{};
            \vertex (C4) at (2,2)[label=above:$C4$]{};
            \wvertex (R2) at (0,2)[label=above:$R2$]{};
            \wvertex (R1) at (4,2)[label=above:$R1$]{};
            \vertex (C3) at (4,0)[label=below:$C3$]{};
            \vertex (C2) at (1,1)[label=below:$C2$]{};
            \path
            	(C1) edge (R2) edge (R3)
            	(C2) edge (R2) edge (R3)
            	(C4) edge (R1) edge (R2) edge (R3)
            	(C3) edge (R1) edge (R3)
            ;
        \end{tikzpicture}\\
        \begin{tikzpicture}
        	\vertex (C1) at (0,0)[label=below:$C1$]{};
        	\vertex (C2) at (2,2)[label=right:$C2$]{};
        	\vertex (C3) at (-2,2)[label=left:$C3$]{};
        	\wvertex (R2) at (-1,1)[label=left:$R2$]{};
        	\wvertex (R1) at (1,1)[label=right:$R1$]{};
        	\wvertex (R3) at (0,2)[label=above:$R3$]{};
        	\vertex (C4) at (0,1)[label=below:$C4$]{};
            \path
            	(R1) edge (C1) edge (C2) edge (C4)
            	(R2) edge (C1) edge (C3) edge (C4)
            	(R3) edge (C2) edge (C3) edge (C4)
            ;
        \end{tikzpicture}&
        \begin{tikzpicture}
        	\wvertex (R3) at (0,0)[label=left:$R3$]{};
        	\vertex (C2) at (1,-1)[label=below:$C2$]{};
        	\vertex (C3) at (1,1)[label=above:$C3$]{};
        	\wvertex (R1) at (2,1)[label=above:$R1$]{};
        	\wvertex (R2) at (2,-1)[label=below:$R2$]{};
        	\vertex (C1) at (3,0)[label=right:$C1$]{};
            \path
                (R3) edge (C2) edge (C3)
                (R1) edge (C1) edge (C2) edge (C3)
                (R2) edge (C1) edge (C2) edge (C3)
            ;
        \end{tikzpicture}\\
        \multicolumn{2}{c}{
        \begin{tikzpicture}
            \wvertex (R1) at (0,-1)[label=below:$R1$]{};
            \wvertex (R2) at (0,1)[label=above:$R2$]{};
            \vertex (C1) at (-1,0)[label=right:$C1$]{};
            \vertex (C2) at (-2,0)[label=left:$C2$]{};
            \vertex (C3) at (1,0)[label=left:$C3$]{};
            \vertex (C4) at (2,0)[label=right:$C4$]{};
            \path
            	(R1) edge (C1) edge (C2) edge (C3) edge (C4)
            	(R2) edge (C1) edge (C2) edge (C3) edge (C4)
            ;
        \end{tikzpicture}
        }
    \end{tabular}
    \caption{
        The seven possibilities (up to bipartite isomorphism) for $\core_2(G_E)$ when $E$ is a basis of $X_{3,r}$.}\label{fig:sevenPossible}
    \end{figure}
    \color{black}

    By relabeling vertices, we may assume that $\core_2(G_E)$ is supported on partite vertex sets $1,\dots,n'$ and $1,\dots,r'$.
    Let $A_{\core_2(G_E)}$ denote the incidence matrix of $\core_2(G_E)$
    and let $M_{\core_2(G_E)}$ denote the matrix whose columns are a basis of the kernel of $A_{\core_2(G_E)}$.
    Then $\cc\{v_{C_i}: i = 1,2,3\}$ is the span of $M_{\core_2(G_E)}$.
    Letting $K'$ be the submatrix of $K$ with columns corresponding to the edge set of $\core_2(G_E)$,
    we are done if we show that $K'M_{\core_2(G_E)}$ has rank 3 for the seven values of $\core_2(G_E)$ above.
    This is verified in a Mathematica script available at the following url\cite{Mathematica}.
\begin{quote}
\url{https://dibernstein.github.io/Supplementary_materials/funtf.html}
\end{quote}
{This Mathematica script computes each of the symbolic matrices $K'M_{\core_2(G_E)}$,
then calculates their rank as matrices with entries in the appropriate function field.}
\end{proof}

The following proposition is useful to construct examples showing that the converse of Theorem \ref{thm:mustBeConnected} is not true for $n \ge 4$.

\begin{prop}\label{prop:atMostTwoUnknownColumns}
	Assume $r \ge n+2 > 4$ and
	let $E \subseteq [n]\times [r]$. 
    If $E$ is spanning in $\Xnr$,
    then at most two vertices of $G_E$ corresponding to columns can have degree $n$.
\end{prop}
\begin{proof}
    Assume $G_E$ has $k$ column vertices of degree $n$.
    Without loss of generality, assume they correspond to the first $k$ columns
    so that $(a,i) \notin E$ for all $1 \le a \le n$, $1 \le i \le k$.
    Define $E' := \{(a,i): 1 \le a \le n, k+1 \le i \le r\}$, and observe
    $E \subseteq E'$.
    We show that {when $k \ge 3$,} the dimension of $\pi_{E'}^{-1}(\pi_{E'}(M))$ is positive for generic $M$.
    It follows that $E'$, and therefore $E$, is not spanning.

    Let $M \in \Xnr$ be a generic finite unit norm tight frame.
    The $(i,a)$ entry of $M$ will be denoted $m_{ia}$.
    Let $\tilde g_a$, $\tilde f_{ij}$ denote the polynomials obtained from $g_a$ 
    and $f_{ij}$,
    as in \eqref{eqG} and $\eqref{eqF}$,
    by plugging in $m_{ia}$ for $x_{ia}$ when $(i,a) \in E'$.
    The Zariski closure of $\pi_{E'}^{-1}(\pi_{E'}(M))$ can be identified with the variety in $\cc^{n\times k}$
    defined by the polynomials $\tilde g_1,\dots,\tilde g_k$ and $\tilde f_{ij}$ for $1 \le i \le j \le n$.
    
    {Following from Equation~\ref{eq:fgRelation}, we have 
    \begin{align}\label{eq:dropG}
    	\tilde f_{11} + \tilde f_{22} + \dots + \tilde f_{nn} = \tilde g_1 + \dots + \tilde g_k 
    \end{align}
    }
    and so $\pi_{E'}^{-1}(\pi_{E'}(M))$ is in fact the vanishing locus of $\tilde g_1,\dots,\tilde g_{k-1}$ and $\tilde f_{ij}$ for $1 \le i,j \le n$.
    Moreover, the polynomials
    \begin{align}\label{eq:symmetricMatrixLowRank}
    	\tilde f_{ij} + \delta_{ij}\frac{r}{n} - \sum_{a = k+1}^r m_{ia}m_{ja} \qquad {\rm where} \qquad \delta_{ij} :=
    	\begin{cases}
    	    1 & i = j\\
    	    0 & i \neq j
    	\end{cases}
    \end{align}
    parameterize the variety of $n\times n$ symmetric matrices of rank at most $k$,
    which has dimension $nk-\binom{k}{2}$ (see for example \cite[Lemma~6.2]{bernstein2020typical}).
    Thus the $\binom{n+1}{2}$ polynomials $\tilde f_{ij}, 1 \le i\le j \le n$
    together contribute at most $nk-\binom{k}{2}$ to the codimension of $\pi_{E'}^{-1}(\pi_{E'}(M))$.
    Hence the codimension of $\pi_{E'}^{-1}(\pi_{E'}(M))$ is at most $nk-\binom{k}{2} + k-1$,
    which is strictly less than $nk$ for $k \ge 3$.
\end{proof}

\begin{example}\label{ex:converseConnectivityFalse}
    Let $n \ge 4$ and $r \ge n+2$ and let $E' := \{(i,a): 2 \le i \le n, 4 \le a \le r\}$.
    Let $E$ be obtained from $E'$ by removing any $\binom{n-2}{2}-1$ elements.
    Then $E$ has $nr - \binom{n+1}{2} - r + 1$ elements and $G_E$ is connected.
    However, Proposition \ref{prop:atMostTwoUnknownColumns} implies that $E$ cannot be a basis of $\Xnr$.
    Figure \ref{fig:connectedNonBasis} shows examples of this construction for $n = 4$ and $n = 5$.

    \begin{figure}[htb!]
    \begin{equation*}
        \begin{pmatrix}
            0&0&0&0&0&0\\
            0&0&0&1&1&1\\
            0&0&0&1&1&1\\
            0&0&0&1&1&1
        \end{pmatrix}
        \qquad
        \begin{pmatrix}
            0&0&0&0&0&0&0\\
            0&0&0&1&1&1&1\\
            0&0&0&1&1&0&1\\
            0&0&0&1&1&1&0\\
            0&0&0&1&1&1&1
        \end{pmatrix}
    \end{equation*}
    \caption{Matrix entry representations of examples of the construction given in Example \ref{ex:converseConnectivityFalse}
    for $n = 4$ and $n=5$.
    These show that when $n \ge 4$,
    a graph on $nr-\binom{n+1}{2}-r+1$ edges whose bipartite complement is connected
    may fail to be a basis of $\Xnr$.
    }\label{fig:connectedNonBasis}
    \end{figure}
\end{example}

\subsection{Combinatorial criteria with fixed row size}
The goal of this section is to fix $\nn$ and find a combinatorial criteria to determine if $E$ is a basis of $\Xnr$ for any $\rr$.
This is made precise in Remark~\ref{remark:punchline}.

The following theorem tells us that whether or not a given $E \subseteq [n]\times [r]$ of cardinality $nr-\binom{n+1}{2}-r+1$
is a basis in $\Xnr$ depends only on $\core_2(G_E)$.

\begin{thm}\label{theorem:onlyTwoCoreMatters}
	Assume $r \ge n+2 > 4$ and let $E \subseteq [n]\times [r]$ such that $|E| = nr-\binom{n+1}{2} - r + 1$.
    Then $E$ is a basis in the algebraic matroid underlying $\Xnr$ if and only if
    the set $E' \supseteq E$ satisfying $G_{E'} = \gcore_2(G_E)$ is spanning in $\Xnr$.
    Moreover, for a fixed $n$, there are only finitely many possible graphs appearing as $\core_2(G_E)$
    as $E$ ranges over all bases of $\Xnr$.
\end{thm}

\begin{proof}
	Let $M \in \Xnr$ be a generic finite unit norm tight frame whose $(i,a)$ entry is $m_{ia}$.
	Let $\tilde g_a$ and $\tilde f_{ij}$ denote the polynomials obtained from $g_a$ and $f_{ij}$
	by setting $x_{ia} = m_{ia}$ when $(i,a) \in E$.
	Then $\pi_E^{-1}(\pi_E(M))$ can be viewed as the zero-dimensional variety in $\cc^{[n]\times[r]\setminus E}$ defined
	by the vanishing of the polynomials $\tilde g_a$, $1 \le a \le r$ and $\tilde f_{ij}$, $1 \le i \le j \le n$.
	Since $G_E$ is connected, the edges of $G_E$ that are not in $\gcore_2(G_E)$ can be ordered $(i_1,a_1),\dots,(i_k,a_k)$
	such that for each $j$,
	in either $\tilde g_{a_j}$ or $\tilde f_{i_j i_j}$,
	every variable other than $x_{i_j a_j}$ that appears is of the form $x_{i_l a_l}$ for some $l < j$.
	It follows that given $\pi_E(M)$,
	one can solve a series of quadratic equations in order to recover, up to finite ambiguity,
	the entries of $M$ at positions corresponding to edges of $G_E$ that are not in $\gcore_2(G_E)$.
	One can then solve for the remaining entries of $M$ precisely when $\gcore_2(G_E)$ is spanning in $\Xnr$.
    The ``moreover'' clause follows by Proposition~\ref{prop:technicalStuff} below.
\end{proof}

\begin{remark}\label{remark:punchline}
Given a set $E \subseteq [n]\times [r]$ that is spanning in $\Xnr$,
the set $E \cup \{(1,r+1),\dots,(n,r+1)\}$ is spanning in $X_{n,r+1}$.
Thus Theorem \ref{theorem:onlyTwoCoreMatters} tells us that if we fix $n$ but allow $r$ to vary,
then the problem of determining whether or not $E \subset [n]\times [r]$ is a basis of $\Xnr$
is equivalent to determining whether or not $\core_2(G_E)$ appears on a certain finite list.
Proposition~\ref{prop:technicalStuff} below gives us the finiteness statement in Theorem \ref{theorem:onlyTwoCoreMatters},
as well as bounds on the size of $\core_2(G)$.
\end{remark}

\begin{prop}\label{prop:technicalStuff}
	Let $r \ge n+2 > 4$ and let $E \subseteq [n]\times [r]$ such that $G_E$ is connected.
    Let $\alpha$ and $\beta$ be the number of row- and column-vertices (respectively) in $\core_2(G_E)$.
    If $E$ is a basis of $\Xnr$, then
    \begin{enumerate}
        \item $\alpha = n-1$ or $\alpha = n$
        \item $\alpha \le \beta \le \binom{n}{2} + \alpha - 1$.
    \end{enumerate}
\end{prop}
\begin{proof}
	Let $M \in \Xnr$ be a generic finite unit norm tight frame
	 whose $(i,a)$ entry is $m_{ia}$.
	Let $\tilde g_a$ and $\tilde f_{ij}$ denote the polynomials obtained from $g_a$ and $f_{ij}$
	by setting $x_{ia} = m_{ia}$ when $(i,a) \in E$.
	Then $\pi_E^{-1}(\pi_E(M))$ can be viewed as the zero-dimensional variety in $\cc^{[n]\times[r]\setminus E}$ defined
	by the vanishing of the polynomials $\tilde g_a$, $1 \le a \le r$ and $\tilde f_{ij}$, $1 \le i \le j \le n$.

	First we show $\alpha=n-1$ or $\alpha=n$.
    Without loss of generality,
    assume that $\core_2(G_E)$ has row-vertices $1,\dots,\alpha$ and column vertices $1,\dots,\beta$.
    Let $F := \{(i,a): 1 \le i \le \alpha, 1 \le a \le \beta\} \setminus E$ be the edge set of $\gcore_2(G_E)$. 
    {This set has cardinality given by  $|F| = \alpha + \beta + \binom{n}{2} - 1$.}
    The elements of $F$ index the entries in the upper-left submatrix of $M$ that are,
    in principle,
    allowed to vary over the fiber $\pi_E^{-1}(\pi_E(M))$.
    After dropping one of the redundant $\tilde g_a$'s via \eqref{eq:dropG},
    there are exactly $\binom{\alpha + 1}{2} + \beta - 1 + \alpha(n-\alpha)$
    equations among the $\tilde g_a$'s and $\tilde f_{ij}$'s that involve entries in the upper-left $\alpha \times \beta$ block of $M$.
    Since $\pi_E^{-1}(\pi_E(M))$ is zero-dimensional,
    we must have
    \begin{equation*}
        \alpha + \beta + \binom{n}{2} - 1 \le \binom{\alpha + 1}{2} + \beta - 1 + \alpha(n-\alpha)
    \end{equation*}
    which simplifies to
    \begin{equation}\label{eq:abInequality}
        -\frac{1}{2}\alpha^2 + \left(n-\frac{1}{2}\right)\alpha - \binom{n}{2} \ge 0.
    \end{equation}

    Let us now consider the left-hand side of the inequality \eqref{eq:abInequality} as a polynomial $h$ in $\alpha$,
    treating $n$ as a constant.
    The only roots of $h$ are $n-1$ and $n$, and $h(\alpha)$ is nonnegative if and only if $n-1\le\alpha\le n$.
    Therefore, we must have $\alpha\in\{n-1,n\}$.

    Now we show $\alpha \le \beta$.
    As noted in the proof of Proposition \ref{prop:atMostTwoUnknownColumns},
    when $(i,a)$ is an edge in $G_E$ but not $\gcore_2(G_E)$,
    we may solve a zero-dimensional quadratic system for $x_{ia}$ given $\{m_{ia}: (i,a) \in E\}$.
    Thus we may now assume that $G_E = \gcore_2(G_E)$ and allow $E$ to be spanning in $\Xnr$ (as opposed to a basis of $\Xnr$).

    Assume for the sake of contradiction that $\alpha > \beta$.
    Then,
     the $\binom{\alpha+1}{2}$ constraints $\tilde f_{ij} = 0$ where $1 \le i \le j \le \alpha$
    together can contribute at most $\alpha\beta-\binom{\beta}{2}$ to codimension.
    This is because for $1 \le i \le j \le \alpha$,
    the polynomials from \eqref{eq:symmetricMatrixLowRank} with $k = \beta$
    give the entries of an $\alpha\times \alpha$ symmetric matrix with rank at most $\beta$,
    and the dimension of the variety of $\alpha\times\alpha$ symmetric matrices of rank at most $\beta$ is $\alpha\beta-\binom{\beta}{2}$
    (see for example \cite[Lemma~6.2]{bernstein2020typical}).
    Also, as before, at least one of the constraints $\tilde g_a = 0$, $1 \le a \le \beta$ is redundant.
    Since $\pi_E^{-1}(\pi_E(M))$ is zero-dimensional,
    we must have $|F| \le \alpha\beta-\binom{\beta}{2} + \beta -1$
    and therefore
    \begin{equation}\label{eq:alphaBetaInequality}
    	\alpha(\beta-1) - \binom{\beta}{2} - \binom{n}{2} \ge 0.
    \end{equation}
    After plugging in $n-1$ for $\alpha$,  \eqref{eq:alphaBetaInequality} becomes $(\beta - n)^2 + n + \beta \le 2$,
    which is a contradiction because $n \ge 3$.
    Plugging in $n$ for $\alpha = n$ in
    \eqref{eq:alphaBetaInequality}, we get the inequality $-n-\beta \ge (\beta - n)^2$,
    which is a contradiction because the left hand side is strictly negative, and the right hand side is nonnegative.
    Hence, we have $\alpha\le \beta$.

    The final inequality $\beta \le \binom{n}{2} + \alpha - 1$ follows from the fact that $\core_2(G_E)$ has $\binom{n}{2} + \alpha + \beta - 1$ edges
    and each of the $\beta$ non-isolated column vertices has degree at least~$2$.
\end{proof}

\section{Degree of projection and algebraic identifiability }\label{sec:degrees}

Now that we have a handle on which subsets $E \subseteq [n]\times [r]$
yield projections $\pi_E: \Xnr \rightarrow \cc^E$ that are generically finite-to-one,
we can ask about the cardinality of a generically finite fiber.
In other words, we want to solve the following problem.
\begin{problem}[Algebraic identifiability complexity]\label{pr:algIdentifiable}
Develop a combinatorial method for computing the degree of the map $\pi_E:\Xnr \rightarrow \cc^E$
from $G_E$ when $E$ is a basis of $\Xnr$.
\end{problem}
The following theorem gets us part of the way towards a solution to Problem~\ref{pr:algIdentifiable}.

\newcommand{\clGE}{{\bar E}}
\begin{thm}\label{thm:degreeOfProjection}
    Let $r \ge n+2 > 4$ and let $E \subseteq [n]\times [r]$.
    Define $F \subseteq[n]\times [r]$  such that $\gcore_2(G_E) = G_F$,
    and let $k$ denote the number of vertices that are isolated in $\gcore_2(G_E)$ but not in $G_E$.
	If $E$ is a spanning set of $\Xnr$, then we have
	\[
		\deg \pi_E=2^k \cdot\deg \pi_F.
	\]
\end{thm}
\begin{proof}
	First observe, $E \subseteq F$ and so we have a projection map $h: \pi_F(\Xnr) \rightarrow \cc^E$
	that omits all the coordinates corresponding to elements of $F \setminus E$.
	Then we have $\pi_E= h\circ\pi_F$ and
    \begin{equation*}
    	\deg \pi_E=\deg \pi_F \cdot \deg h,
    \end{equation*}
    which can be seen as follows.
    The maps $h: \pi_F(\Xnr) \rightarrow \cc^E$ and $\pi_E: \Xnr \rightarrow \cc^E$ are each branched covers of $\cc^E$.
     In other words, there exist dense Zariski open subsets $U_1$ and $U_2$ of $\cc^E$ such that $h$ restricted to $h^{-1}(U_1)$ and $\pi_E$ restricted to $\pi_E^{-1}(U_2)$ are covering spaces.
     Moreover,  $h$ and $\pi_E$ restricted to $U_1\cap U_2$ are also covering spaces.
     Since $\pi_E= h\circ\pi_F$, we have $\pi_F$ is a homomorphism of covering spaces and thus
     the topological degree of the restricted maps with image $U_1\cap U_2$ satisfy $\deg \pi_E=\deg \pi_F \cdot \deg h$. As $U_1$ and $U_2$ are  dense Zariski open subsets of $\cc^E$, we have $U_1\cap U_2$ is also a dense Zariski open subset of $\cc^E$. Therefore, the equality above follows.

    It now suffices to show $\deg h = 2^k$.
    Let $M \in \Xnr$ be generic and
    define $\tilde g_a$ and $\tilde f_{ij}$ as in the proofs of Theorem \ref{theorem:onlyTwoCoreMatters} and Proposition \ref{prop:technicalStuff}.
    We can order the elements of $F\setminus E$ as $(i_1,a_1),\dots,(a_k,i_k)$
    such that for each $j \in \{1,\dots,k\}$,
    the only non-$x_{i_j a_j}$ variables in at least one of $g_{a_j}$ or $f_{i_ji_j}$ will be of the form $x_{i_la_l}$ with $l < j$.
{
    Let $\calf$ denote the system of all such polynomials.
    The non-constant coordinates of $h^{-1}(\pi_E(M))$ are given by the variety defined by the vanishing of $\calf$.
    }
    By solving $\calf$ via ``back-substitution'' in the order $x_{a_1i_1},\dots,x_{a_ki_k}$,
    we see that this variety has exactly $2^k$ points.
    Thus, $| h^{-1}(f(M))|=2^k$ and $\deg h=2^k$.
\end{proof}

Let $E \subseteq [n]\times [r]$ with $E$ a basis of $\Xnr$
and define $F \subseteq [n] \times [r]$ so that  $G_F = \gcore_2(G_E)$.
Let $\beta$ denote the number of column vertices in $\core_2(G_E)$
and let $M \in \Xnr$ be generic whose $(i,a)$ entry is $m_{ia}$.
If $r \ge n+2 > 4$ and $r \ge \beta+1$,
then the degree of the projection map $\pi_F: \Xnr \rightarrow \cc^F$ only depends on $\core_2(G_E)$ and not on $r$.
This follows from the fact that if $r \ge \beta+1$,
then the set of non-constant polynomials $\tilde f_{ij}$ and $\tilde g_a$ obtainable by substituting $x_{ia} = m_{ia}$
for $(i,a) \in F$ does not depend on $r$.
So for a graph $H$ such that $H = \core_2(G_E)$ for some basis $E$ of $\Xnr$,
let $\deg(H)$ denote the degree of $\pi_F$ when $r \ge \beta+1$.
{
If $r = \beta$ ($r < \beta$ is not possible),
then $\deg(\pi_F) \le \deg(\core_2(G_E))$.}
Thus Theorem \ref{thm:degreeOfProjection} gives us the bound $\deg(\pi_E) \le 2^k \deg(\core_2(G_E))$.

Theorem \ref{theorem:onlyTwoCoreMatters} tells us, that for fixed $n$,
there are only \emph{finitely many} $\core_2(G_E)$.
Thus one can compute all values of $\deg(\core_2(G_E))$ for a fixed $n$
and use this to produce an algorithm that bounds the size of a finite fiber $|\pi_E^{-1}(\pi_E(M))|$
by computing the $2$-core of $G_E$. This is done in Algorithm~\ref{alg:fiberSize}
and Example \ref{ex:finiteFiberN3} illustrates this for the case $n = 3$.

\begin{algorithm}[H]
\caption{For fixed $n \ge 3$, bounds the size of a generic fiber $\pi_E^{-1}(\pi_E(M))$ when $E$ is a basis of $\Xnr$.
Assumes that all possible values of $\deg(\core_2(G_E))$ have been precomputed.}\label{alg:fiberSize}
\begin{algorithmic}[1]
	\Procedure{BoundFiber}{$r,E$} \Comment{$r \ge n+2$ and $E$ is a basis of $\Xnr$}
	\State{$H \gets \core_2(G_E)$}
	\State{$k \gets$ number of vertices in $G_E$ but not $H$}
	\State{$d \gets \deg(H)$}\Comment{obtain by looking up in precomputed table}
	\State{\textbf{Return:} $d\cdot 2^k$}
	\EndProcedure
\end{algorithmic}
\end{algorithm}

\begin{example}\label{ex:finiteFiberN3}
When $E$ is a basis of $X_{3,r}$,
$\core_2(G_E)$ is one of seven graphs, displayed in Figure \ref{fig:sevenPossible}.
For each possible $\core_2(G_E)$, we compute the cardinality of a projection of an $X_{3,r}$ onto $\gcore_2(G_E)$
using probability-one methods in \texttt{Bertini}~\cite{BHSW06} via \texttt{Macaulay2}~\cite{M2,bertini4M2}
When $\core_2(G_E)$ has five or more column vertices, we took $r = 5$ and $r = 6$
and observed that in both cases, the degree of projection was the same.
When $\core_2(G_E)$ has fewer than five column vertices, we take $r = 5$.
These degrees are given in Table~\ref{tab:degreesOf2Core}.
Via the above discussion, this characterizes the possible degrees of a projection of $X_{3,r}$ onto a basis.
For example, if $E \subseteq [3]\times[5]$
where as a zero-one matrix,
\[
	E=\begin{pmatrix}
        0&0&0&0&0\\
        1&1&0&0&0\\
        1&1&1&0&0\\
    \end{pmatrix},
\]
then the degree of a generic fiber of this projection is $128=2^2\cdot 32$.
This can be read off from Table~\ref{tab:degreesOf2Core}
by noting that $F$ such that $G_F = \core_2(G_E)$ is given in the top row of the table and the degree of the corresponding fiber is $32$.

\begin{table}[htb!]
    \begin{center}
        \begin{tabular}{|c |c |}
        \hline
        $F$ such that $G_F = \core_2(G_E)$ & $|\deg(\core_2(G_E))|$ \\
        \hline
            $\begin{pmatrix}
                 0&0&0\\
                 0&0&0\\
                 1&0&0\\
            \end{pmatrix}$ & 32 \\\hline
            $\begin{pmatrix}
                 0&0&0&0\\
                 0&0&0&0\\
            \end{pmatrix}$ &24 \\\hline
            $\begin{pmatrix}
                 0&0&0&0\\
                 1&1&0&0\\
                 0&0&1&0\\
            \end{pmatrix}$ &96\\\hline
            $\begin{pmatrix}
                 0&0&1&0\\
                 0&1&0&0\\
                 1&0&0&0\\
            \end{pmatrix}$ &128 \\\hline
            $\begin{pmatrix}
                 1&1&1&0&0\\
                 0&0&0&1&1\\
                 0&0&0&0&0\\
            \end{pmatrix}$ &288 \\\hline
            $\begin{pmatrix}
                 0&0&1&1&0\\
                 1&1&0&0&0\\
                 0&0&0&0&1\\
            \end{pmatrix}$ &576 \\\hline
            $\begin{pmatrix}
                 0&0&1&1&1\\
                 0&1&0&0&0\\
                 1&0&0&0&0\\
            \end{pmatrix}$ &384 \\
        \hline
    \end{tabular}
    \vspace{5pt}\caption{Each possible value of $F \subseteq [3]\times [5]$ such that $G_F = \core_2(G_E)$ when $E$ is a basis of $X_{3,r}$,
    alongside the degree of the corresponding coordinate projection map.}\label{tab:degreesOf2Core}
    \end{center}
\end{table}
\end{example}

\section*{Acknowledgments}
We are very thankful for the reviewers' comments to improve this article. 
This work was partially supported by the grants NSF CCF-1708884 and NSA H98230-18-1-0016.
Daniel Irving Bernstein was supported by an NSF Mathematical Sciences Postdoctoral Research Fellowship (DMS-1802902).
He also completed some of this work while employed by the NSF-supported (DMS-1439786)
Institute for Computational and Experimental Research in Mathematics in Providence, RI,
during the Fall 2018 semester program on nonlinear algebra.
Cameron Farnsworth's research was supported in part by the National Research Foundation of Korea (Grant Number 20151009350).
Jose Israel Rodriguez was partially supported by the College of Letters and Science, UW-Madison.

\bibliographystyle{abbrv}
\bibliography{framebib}{}

\begin{thebibliography}{10}

\bibitem{bertini4M2}
D.~J. Bates, E.~Gross, A.~Leykin, and J.~I. Rodriguez.
\newblock {Bertini for Macaulay2}.
\newblock {\em preprint arXiv:1310.3297}, 2013.

\bibitem{BHSW06}
D.~J. Bates, J.~D. Hauenstein, A.~J. Sommese, and C.~W. Wampler.
\newblock Bertini: Software for numerical algebraic geometry.
\newblock Available at bertini.nd.edu with permanent doi:
  dx.doi.org/10.7274/R0H41PB5.

\bibitem{bernstein2017completion}
D.~I. Bernstein.
\newblock Completion of tree metrics and rank 2 matrices.
\newblock {\em Linear Algebra and its Applications}, 533:1--13, 2017.

\bibitem{bernstein2018matroids}
D.~I. Bernstein.
\newblock {\em Matroids in Algebraic Statistics.}
\newblock PhD thesis, North Carolina State University, 2018.
\newblock \url{https://repository.lib.ncsu.edu/handle/1840.20/35009}.

\bibitem{bernstein2019typical}
D.~I. Bernstein, G.~Blekherman, and K.~Lee.
\newblock Typical ranks in symmetric matrix completion.
\newblock {\em arXiv preprint arXiv:1909.06593}, 2019.

\bibitem{bernstein2020typical}
D.~I. Bernstein, G.~Blekherman, and R.~Sinn.
\newblock Typical and generic ranks in matrix completion.
\newblock {\em Linear Algebra and its Applications}, 585:71--104, 2020.

\bibitem{blekherman2019maximum}
G.~Blekherman and R.~Sinn.
\newblock Maximum likelihood threshold and generic completion rank of graphs.
\newblock {\em Discrete \& Computational Geometry}, 61(2):303--324, 2019.

\bibitem{borcea2004number}
C.~Borcea and I.~Streinu.
\newblock The number of embeddings of minimally rigid graphs.
\newblock {\em Discrete \& Computational Geometry}, 31(2):287--303, 2004.

\bibitem{BBBKR2017}
M.~Brandt, J.~Bruce, T.~Brysiewicz, R.~Krone, and E.~Robeva.
\newblock The degree of {${\rm SO}(n,\Bbb C)$}.
\newblock In {\em Combinatorial algebraic geometry}, volume~80 of {\em Fields
  Inst. Commun.}, pages 229--246. Fields Inst. Res. Math. Sci., Toronto, ON,
  2017.

\bibitem{CFMPSframes}
J.~Cahill, M.~Fickus, D.~G. Mixon, M.~J. Poteet, and N.~Strawn.
\newblock Constructing finite frames of a given spectrum and set of lengths.
\newblock {\em Appl. Comput. Harmon. Anal.}, 35(1):52--73, 2013.

\bibitem{CMSframes}
J.~Cahill, D.~G. Mixon, and N.~Strawn.
\newblock Connectivity and irreducibility of algebraic varieties of finite unit
  norm tight frames.
\newblock {\em SIAM J. Appl. Algebra Geom.}, 1(1):38--72, 2017.

\bibitem{CSalgframes}
J.~Cahill and N.~Strawn.
\newblock Algebraic geometry and finite frames.
\newblock In {\em Finite frames}, Appl. Numer. Harmon. Anal., pages 141--170.
  Birkh\"auser/Springer, New York, 2013.

\bibitem{capco2018number}
J.~Capco, M.~Gallet, G.~Grasegger, C.~Koutschan, N.~Lubbes, and J.~Schicho.
\newblock The number of realizations of a laman graph.
\newblock {\em SIAM Journal on Applied Algebra and Geometry}, 2(1):94--125,
  2018.

\bibitem{CFMWZfusionframes}
P.~G. Casazza, M.~Fickus, D.~G. Mixon, Y.~Wang, and Z.~Zhou.
\newblock Constructing tight fusion frames.
\newblock {\em Appl. Comput. Harmon. Anal.}, 30(2):175--187, 2011.

\bibitem{CKframeerasures}
P.~G. Casazza and J.~Kova{\v c}evi\'c.
\newblock Equal-norm tight frames with erasures.
\newblock {\em Adv. Comput. Math.}, 18(2-4):387--430, 2003.
\newblock Frames.

\bibitem{DSframeconditions}
R.~J. Duffin and A.~C. Schaeffer.
\newblock A class of nonharmonic {F}ourier series.
\newblock {\em Trans. Amer. Math. Soc.}, 72:341--366, 1952.

\bibitem{DykemaStrawn}
K.~Dykema and N.~Strawn.
\newblock Manifold structure of spaces of spherical tight frames.
\newblock {\em Int. J. Pure Appl. Math.}, 28(2):217--256, 2006.

\bibitem{emiris2009algebraic}
I.~Z. Emiris, E.~P. Tsigaridas, and A.~E. Varvitsiotis.
\newblock Algebraic methods for counting euclidean embeddings of rigid graphs.
\newblock In {\em International Symposium on Graph Drawing}, pages 195--200.
  Springer, 2009.

\bibitem{FWWgeneratingframes}
D.-J. Feng, L.~Wang, and Y.~Wang.
\newblock Generation of finite tight frames by {H}ouseholder transformations.
\newblock {\em Adv. Comput. Math.}, 24(1-4):297--309, 2006.

\bibitem{MR3473144}
M.~Fickus, J.~D. Marks, and M.~J. Poteet.
\newblock A generalized {S}chur-{H}orn theorem and optimal frame completions.
\newblock {\em Appl. Comput. Harmon. Anal.}, 40(3):505--528, 2016.

\bibitem{FMPSframes}
M.~Fickus, D.~G. Mixon, M.~J. Poteet, and N.~Strawn.
\newblock Constructing all self-adjoint matrices with prescribed spectrum and
  diagonal.
\newblock {\em Adv. Comput. Math.}, 39(3-4):585--609, 2013.

\bibitem{Goyal}
V.~K. Goyal.
\newblock {\em Beyond traditional transform coding}.
\newblock PhD thesis, University of California, Berkeley, 1998.

\bibitem{GKKframeerasures}
V.~K. Goyal, J.~Kova{\v c}evi\'c, and J.~A. Kelner.
\newblock Quantized frame expansions with erasures.
\newblock {\em Appl. Comput. Harmon. Anal.}, 10(3):203--233, 2001.

\bibitem{GVTovercomplete}
V.~K. Goyal, M.~Vetterli, and N.~T. Thao.
\newblock Quantized overcomplete expansions in {$\bold R^N\colon$} analysis,
  synthesis, and algorithms.
\newblock {\em IEEE Trans. Inform. Theory}, 44(1):16--31, 1998.

\bibitem{M2}
D.~R. Grayson and M.~E. Stillman.
\newblock Macaulay2, a software system for research in algebraic geometry.
\newblock Available at http://www.math.uiuc.edu/Macaulay2/.

\bibitem{HagaPegel}
T.~Haga and C.~Pegel.
\newblock Polytopes of eigensteps of finite equal norm tight frames.
\newblock {\em Discrete Comput. Geom.}, 56(3):727--742, 2016.

\bibitem{HolmesPaulsen}
R.~B. Holmes and V.~I. Paulsen.
\newblock Optimal frames for erasures.
\newblock {\em Linear Algebra Appl.}, 377:31--51, 2004.

\bibitem{Horn}
A.~Horn.
\newblock Doubly stochastic matrices and the diagonal of a rotation matrix.
\newblock {\em American Journal of Mathematics}, 76:620--630, 1954.

\bibitem{Mathematica}
W.~R. Inc.
\newblock Mathematica, {V}ersion 11.3.
\newblock Champaign, IL, 2018.

\bibitem{jackson2012number}
B.~Jackson and J.~Owen.
\newblock The number of equivalent realisations of a rigid graph.
\newblock {\em arXiv preprint arXiv:1204.1228}, 2012.

\bibitem{kiraly2015algebraic}
F.~J. Kir{\'a}ly, L.~Theran, and R.~Tomioka.
\newblock The algebraic combinatorial approach for low-rank matrix completion.
\newblock {\em The Journal of Machine Learning Research}, 16(1):1391--1436,
  2015.

\bibitem{MR2881262}
A.~Leykin.
\newblock Numerical algebraic geometry.
\newblock {\em J. Softw. Algebra Geom.}, 3:5--10, 2011.

\bibitem{MRframecompletions}
P.~G. Massey and M.~A. Ruiz.
\newblock Tight frame completions with prescribed norms.
\newblock {\em Sampl. Theory Signal Image Process.}, 7(1):1--13, 2008.

\bibitem{McKay201494}
B.~D. McKay and A.~Piperno.
\newblock Practical graph isomorphism, \{II\}.
\newblock {\em Journal of Symbolic Computation}, 60(0):94 -- 112, 2014.

\bibitem{ORSfradeco}
L.~Oeding, E.~Robeva, and B.~Sturmfels.
\newblock Decomposing tensors into frames.
\newblock {\em Adv. in Appl. Math.}, 73:125--153, 2016.

\bibitem{oxley2006matroid}
J.~G. Oxley.
\newblock {\em Matroid theory}, volume~3.
\newblock Oxford University Press, USA, 2006.

\bibitem{rosen2014computing}
Z.~Rosen.
\newblock Computing algebraic matroids.
\newblock {\em arXiv preprint arXiv:1403.8148}, 2014.

\bibitem{Schur}
I.~Schur.
\newblock {\"U}ber eine klasse von mittelbildungen mit anwendungen auf die
  determinantentheorie.
\newblock {\em Sitzungsber. Berl. Math. Ges}, 22:9--20, 1923.

\bibitem{StrawnMasterThesis}
N.~Strawn.
\newblock Geometry and constructions of finite frames.
\newblock Master's thesis, Texas A\&M University, College Station, TX USA,
  2007.

\bibitem{Strawnframes}
N.~Strawn.
\newblock Finite frame varieties: nonsingular points, tangent spaces, and
  explicit local parameterizations.
\newblock {\em J. Fourier Anal. Appl.}, 17(5):821--853, 2011.

\bibitem{verschelde1999algorithm}
J.~Verschelde.
\newblock Algorithm 795: Phcpack: A general-purpose solver for polynomial
  systems by homotopy continuation.
\newblock {\em ACM Transactions on Mathematical Software (TOMS)},
  25(2):251--276, 1999.

\end{thebibliography}

\end{document}